\documentclass[12pt]{amsart}
\usepackage[usenames,dvipsnames]{xcolor}
\usepackage{tikz}
\usepackage{fancyhdr}
\usepackage{txfonts}
\usepackage{graphicx}
\usepackage{epsfig}
\usepackage{mathrsfs}
\usepackage{amssymb}
\usepackage{latexsym}
\usepackage{amsmath}
\usepackage{amsfonts}
\usepackage{amsbsy}
\usepackage{amscd}
\usepackage{subfigure}
\usepackage{verbatim}
\usepackage{color,tikz}
\usetikzlibrary{calc}

\usepackage{amsthm}
\usepackage{amsfonts}
\usepackage{amsbsy,calc}
\usepackage{graphicx,ifthen,cite}
\usepackage{bm}
\usepackage[normalem]{ulem}

\newcommand{\D}{{\mathcal D}}

\newcommand{\R}{{\mathbb R}}
\newcommand{\Z}{{\mathbb Z}}

\newcommand{\N}{{\mathbb N}}

\newtheorem{prop}{Proposition}[section]
\newtheorem{lem}[prop]{Lemma}
\newtheorem{ques}[prop]{Question}
\newtheorem{defi}[prop]{Definition}
\newtheorem{coro}[prop]{Corollary}
\newtheorem{theo}[prop]{Theorem}

\newtheorem{exam}[prop]{Example}
\newtheorem{rema}[prop]{Remark}

\newif\ifdraft
\drafttrue
\ifdraft
\else
\fi
\numberwithin{equation}{section}
\title{}
\author{}
\date{}

\begin{document}
	\baselineskip 18pt
	
	\title[Primes are Complete for a Class of Cantor Spectral Pair]{Primes are Complete for a Class of Cantor Spectral Pair}
	\author[Z.-C. Chi]{Zi-Chao Chi}
	\address{School of Mathematics and Statistics, Wuhan University, Wuhan 430072, P. R. China}
	\email{skychizichao@163.com}
	\author[J.-F. Lu]{Jian-Feng Lu}
	\address{School of Mathematics and Statistics, Key Lab NAA-MOE and Hubei Key Lab-Math. Sci., Central China Normal University, Wuhan 430079, P. R. China}
	\email{lujianfengmath@163.com}
	\author[H.-X. Wang]{He-Xiang Wang*}
	\address{School of Mathematics and Statistics, Kashi University, Kashi 844000, P. R. China}
	\email{wanghex9@163.com}
	\thanks{*Corresponding author.}
	\thanks{ }

	\begin{abstract}
		We study the complete number problem for a class of self-similar spectral measures on the real line.  We say that a real number $t$ is complete for a spectral pair $(\mu,\Lambda)$ if $(\mu,t\Lambda)$ is also a spectral pair. {In this paper we consider the partial homogeneous Cantor measure $\mu_{N^r,\mathcal D}$
			and its canonical spectrum.} 
		Our main result fully characterizes all prime powers $p^n$ that are complete for $p>N$. In particular, we obtain a necessary and sufficient condition for an arbitrary prime $p$ to be complete. {We further establish a connection between the problem and certain Diophantine equations}, including the Nagell--Ljunggren equation, and obtain complete number results for a class of related models. As applications, known solutions of the Nagell--Ljunggren equation provide complete and incomplete {prime-power} examples. 
		
	\end{abstract}
	\maketitle
	
	\section{Introduction}
	Let $\mu_{b,\mathcal{D}}$ be the {partial} homogeneous Cantor measure generated by the \textit{iterated function system} (IFS) (see, e.g., \cite{Hut1981})
	\begin{equation*}
		\{\phi_d(x)=b^{-1}(x+d) : d \in \mathcal{D}\},
	\end{equation*}
	where $b, N \ge 2$ are integers and the digit set $\mathcal{D}=\{0, 1, \dots, N-1\}$. The support of $\mu_{b, \D}$ is the set
	\begin{eqnarray*}
		T(b, \D)=\left\{\sum_{k=1}^\infty b^{-k}d_k: \text{all $d_k\in \D$}\right\},
	\end{eqnarray*}
	which is a fractal set.
	This measure can be represented as the infinite convolution of discrete measures in the weak sence:
	\begin{equation}\label{eqmeaweak}
		\mu_{b,\mathcal{D}} = \delta_{b^{-1}\mathcal{D}} \ast \delta_{b^{-2}\mathcal{D}} \ast \dots,
	\end{equation}
	where $\delta_{E} := \frac{1}{\#E} \sum_{e \in E} \delta_e$ denotes the discrete measure on a finite set $E$.
	
	\begin{defi}
		We say that a countable set $\Lambda \subset \mathbb{R}$ is a \textbf{spectrum} for a Borel probability measure $\mu$ if the family of exponential functions $E(\Lambda) := \{ \mathrm{e}^{2\pi \mathrm{i} \lambda x} : \lambda \in \Lambda \}$ forms an orthonormal basis for the Hilbert space $L^2(\mu)$. In this case, $\mu$ is called a \textbf{spectral measure} and $(\mu, \Lambda)$ is referred to as a \textbf{spectral pair}.
	\end{defi}
	
	It is well-established that the measure $\mu_{b,\mathcal{D}}$ is spectral if and only if $N \mid b$ (see, e.g.,\cite{Dai2012,Dai2013,Deng2014, DHL2014,HL2008,JP1998}). Let $b = qN$ for some positive integer $q$. When \(q=1\), \(\mu_{N,\mathcal D}\) is the Lebesgue measure on \([0,1]\), {which admits the only spectrum $\Z$ up to translation.} Therefore, we focus on the case \(q\ge2\), where
	\(\mu_{b,\mathcal D}\) is a singular continuous measure.   Under this condition, $\mu_{b,\mathcal{D}}$ possesses a canonical spectrum:
	\begin{equation}
		\Lambda_{b, q\mathcal{D}}= \bigcup_{n=1}^{\infty} \left( q\mathcal{D} + bq\mathcal{D} + \dots + b^{n-1}q\mathcal{D} \right).
	\end{equation}
	
	A fundamental problem in spectral measure theory is to characterize the scaling factors $t$ that preserve the spectral property of a measure. To facilitate our study, we introduce the following definition regarding the scaling factor $t$:
	\begin{defi}
		Let $(\mu,\Lambda)$ be a given spectral pair. We say that a real number $t$ is \textbf{complete} for $(\mu,\Lambda)$ if $t\Lambda$ is also a spectrum of $\mu$. Otherwise, $t$ is called \textbf{incomplete} for $(\mu,\Lambda)$.
	\end{defi}
	
	It is worth noting that $t$ is also referred to as a  {\it eigenvalue} or {\it scaling number} in some other literature. Compared with Lebesgue measures, Cantor measures exhibit much richer and more interesting phenomena. For instance, the mock Fourier series of continuous functions associated with the spectrum $\Lambda_{4,\{0,1\}}$ converges uniformly on the support with respect to (w.r.t.) the spectral measure $\mu_{4,\{0,2\}}$ \cite{Str2006}. But Dutkay, Han and Sun\cite{DHS2014} proved that there exist continuous functions whose mock Fourier series associated with the spectrum \(t\Lambda_{4,\{0,1\}}\) diverges at \(0\) for certain prime complete numbers, including \(t=17,23,29,\ldots\). These connections have motivated extensive study of the complete number problem for various classes of spectral measures (see, for example, \cite{Ai2021,ADH2022,CHW2025,Dai2016,Dutkay2016,DK18,Fu2017,He2019,JKS2011,JLW2024,Kong2025,LL2021,Lu2026,LSW2023,WDA2019,WZ2018, YZ2025,ZXA2026, Kong2025}).
	
	Indeed, even a prime number may not be complete for a given spectral pair. Dutkay and Kraus \cite{DK18} gave such a counterexample (see Example \ref{exam2.12}), and we provide another one in Example \ref{exam2.13}. We are therefore led to the following general question:
	
	\begin{ques}\label{quesprime}
		Characterize all primes $p$ that are complete for the spectral pair $(\mu_{b,\mathcal{D}},\Lambda_{b,\mathcal{C}})$.
	\end{ques}
	
	In this paper, we completely resolve Question~\ref{quesprime} for the case $b=N^r$, $r\ge2$. When $N=r=2$, the case reduces to the model studied by Dutkay and Haussermann \cite{Dutkay2016}, which states that $p^n$ is complete for the spectral pair $(\mu_{4,\{0,2\}},\Lambda_{4,{0,1}})$ for any prime $p\ge5$ and any $n\in\N$. Specifically, for any prime \(p>N\), whether all powers \(p^n\) are complete is determined exactly by whether $\log_p\frac{N^r-1}{N-1}$ is an integer. In this way, the complete number problem is reduced to a concrete problem in number theory. Throughout the remainder of this paper, unless otherwise specified, the term ``complete number" always refers to a complete number for the spectral pair $\left(\mu_{N^r,\mathcal D},\Lambda_{N^r,N^{r-1}\mathcal D}\right).$ Our main result is the following.
	\begin{theo}\label{theomain}
		Let $p$ be a prime. The following statements hold.
		\begin{enumerate}
			\item[\textup{(1)}]
			If $\log_p \frac{N^r-1}{N-1}\notin\mathbb N$, then $p^n$ is complete for any $p>N$ and $n\in\mathbb N$.
			\item[\textup{(2)}]
			If $\log_p \frac{N^r-1}{N-1}\in\mathbb N$, then $p^n$ is complete for any positive integer $n<\log_p \frac{N^r-1}{N-1}$ and incomplete for any integer $n\ge\log_p \frac{N^r-1}{N-1}$.
		\end{enumerate}
	\end{theo}
	\begin{rema}
		We assume $p>N$ in the Theorem~\ref{theomain} (1). Actually, the case $p<N$ is {very tricky}. For example, let $N=5$ and $r=2$. Then $2$ and $2^2$ are complete, whereas $2^3$ is incomplete (see Example~\ref{exam527}\textup{(ii)}).
	\end{rema}
	{Although Theorem~\ref{theomain} does not provide a complete classification of prime powers for $p<N$}, the case of a single prime can be characterized without this restriction. The following theorem answers the Question \ref{quesprime}.
	
	\begin{theo}\label{theoprime}
		Let $p$ be a prime. Then $p$ is complete if and only if $p\nmid N$ and $p\neq\frac{N^r-1}{N-1}$.
	\end{theo}
	Theorem~\ref{theomain} reduces the complete number problem to determining
	when a base-\(N\) repunit is a prime power. Specifically, 
	\[
	\log_p\frac{N^r-1}{N-1}\in\mathbb N
	\]
	is equivalent to
	\[
	\frac{N^r-1}{N-1}
	=1+N+\cdots+N^{r-1}
	=p^m
	\]
	for some \(m\in\mathbb N^+\). 
	
	If $m\ge2$ and $r>2$, then $\frac{N^r-1}{N-1}=p^m$ is a special case of the classical Nagell--Ljunggren equation
	\[
	\frac{x^u-1}{x-1}=y^v,
	\qquad
	x>1,\quad y>1,\quad u>2,\quad v\ge2,
	\]
	with \((x,u,y,v)=(N,r,p,m)\). 
	{These are the only three presently known nontrivial positive integer solutions\cite{N1920,L1943,BM2007}, namely:
		\begin{equation*}
			\frac{3^5-1}{3-1}=11^2,\qquad\frac{7^4-1}{7-1}=20^2\qquad\frac{18^3-1}{18-1}=7^3.
	\end{equation*}} It is conjectured that these are the only positive integer solutions,
	although finiteness {is unknown in the general case};
	see \cite{L1943,S1999,
		BM2002,BM2007}. 
	In view of this, Section~\ref{sec5} uses these known solutions to
	analyze the complete number problem for the corresponding self-similar
	spectral measures.
	
	{A particularly important special case occurs when $x=2$, namely $2^u-y^v=1$, which is relative to the exponential Diophantine equation
		\[
		x^u-y^v=1,
		\qquad x,y,u,v\ge2.
		\]
		Catalan conjectured \cite{Catalan1844} in 1844 that the exponential Diophantine equation has the unique solution
		\[
		(x,y,u,v)=(3,2,2,3),
		\] and then it was proved by Mih\u{a}ilescu
		\cite{M2004} and is now known as the Catalan--Mih\u{a}ilescu
		theorem.}
	We now apply this theorem to the case \(N=2\). In this case,
	\[
	\frac{N^r-1}{N-1}=p^m
	\]
	reduces to
	\[
	2^r-p^m=1.
	\]
	Hence, if \(r,m\ge2\), the Catalan--Mih\u{a}ilescu theorem rules out
	such a solution. Therefore,
	\[
	\log_p(2^r-1)\in\mathbb N
	\]
	can occur only when \(m=1\), namely when \(p=2^r-1\). Thus, if
	\(p\neq2^r-1\), then
	\[
	\log_p(2^r-1)\notin\mathbb N.
	\]
	Combining this with Theorem~\ref{theomain}, we obtain the following result.
	
	\begin{theo}\label{theCHW}
		Let $p\neq 2^r-1$ be an odd prime.
		Then for the spectral pair $
		\left( \mu_{2^r,\{0,1\}},
		\Lambda_{2^r,\{0,2^{r-1}\}}\right) $, $p^n$ is complete for any $n\in\mathbb N^+$.
	\end{theo}
	This paper is organized as follows. In Section \ref{secpre}, we prove that {some necessary conditions} for a number to be complete. {For such integers}, we introduce the integer cycles, which serve as the main tool for determining whether a number is complete. We also present definitions and properties of orders and primitive numbers that are needed in the subsequent proofs. Section \ref{secproof} is devoted to the proof of Theorem \ref{theomain}. In Section \ref{sec5}, we present examples and discuss connections with the Nagell--Ljunggren equation.
	
	\section{Preliminary}\label{secpre}
	In this section, we first prove that any complete number must be an integer coprime to $N$. Consequently, it suffices to consider primes that do not divide $N$. Within this setting, integer cycles become an effective tool for studying complete numbers. Furthermore, primitive numbers, as a special subclass of incomplete numbers, play a central role in the proofs presented in this paper.
	
	Let $\mu$ be a Borel probability measure with compact support in
	$\mathbb R$. The Fourier transform of $\mu$ is defined by
	\[
	\widehat{\mu}(\xi)
	=
	\int e^{-2\pi i \xi x}\,d\mu(x).
	\]
	For a function $f$, we denote its zero set by
	\[
	\mathcal Z(f)=\{x:f(x)=0\}.
	\]
	Let $\Lambda\subset\mathbb R$ be a countable set. We say that $\Lambda$
	is an orthogonal set for $\mu$ if
	\[
	E(\Lambda)=\{e^{2\pi i\lambda x}:\lambda\in\Lambda\}
	\]
	is an orthogonal set in $L^2(\mu)$. This is equivalent to
	\[
	\Lambda-\Lambda
	\subset
	\mathcal Z(\widehat{\mu})\cup\{0\}.
	\]
	
	We shall use the following criterion of Jorgensen and Pedersen
	\cite[Lemma 3.3]{JP1998}, which follows from Parseval's identity.
	
	\begin{theo}\label{theo2.1}
		Let $\mu$ be a Borel probability measure with compact support in
		$\mathbb R$, and define
		\[
		Q_\Lambda(\xi)
		=
		\sum_{\lambda\in\Lambda}
		|\widehat{\mu}(\xi+\lambda)|^2,
		\qquad \xi\in\mathbb R.
		\]
		Then the following statements hold.
		\begin{enumerate}
			\item[\textup{(i)}]
			$\Lambda$ is an orthogonal set for $\mu$ if and only if $Q_\Lambda(\xi)\leq 1$ for all $ \xi\in\mathbb R.$
			
			\item[\textup{(ii)}]
			$\Lambda$ is a spectrum for $\mu$ if and only if $
			Q_\Lambda(\xi)\equiv 1$
			for all $\xi\in\mathbb R.$
		\end{enumerate}
		In both cases, $Q_\Lambda$ admits an entire analytic extension to
		$\mathbb C$.
	\end{theo}

	\begin{rema}\label{ruelle}
		Let $(\mu,\Lambda)$ be the spectral pair $\left(\mu_{N^r,\mathcal D},\Lambda_{N^r,N^{r-1}\mathcal D}\right).$ Then the function 
		$Q_\Lambda(\xi)	=
		\sum_{\lambda\in\Lambda}
		|\widehat{\mu}(\xi+\lambda)|^2$
		satisfies
		\[
		Q_\Lambda(\xi)
		=
		\sum_{d\in\mathcal D}
		\left|
		\widehat{\delta}_{\mathcal D}
		\left(
		N^{-r}\bigl(\xi+tN^{r-1}d\bigr)
		\right)
		\right|^2
		Q_\Lambda
		\left(
		N^{-r}\bigl(\xi+tN^{r-1}d\bigr)
		\right).
		\]
		The right-hand side is called the Ruelle transfer operator defined on
		\(Q_\Lambda\).
	\end{rema}

	The following proposition shows that, in the study of complete numbers, it is enough to consider integers coprime to $N$.
	
	\begin{prop}\label{prop2.3}
		If $t\in\mathbb R$ is complete, then $t\in\mathbb Z $ and $	\gcd(t,N)=1.$
	\end{prop}
	
	\begin{proof}
		For the measure $\mu_{N^r,\mathcal D}$, we have
		\[
		\widehat{\mu}_{N^r,\mathcal D}(\xi)
		=
		\prod_{k=1}^{\infty}
		\widehat{\delta}_{\mathcal D}(N^{-kr}\xi)
		\]
		and
		\[
		\mathcal Z(\widehat{\delta}_{\mathcal D})
		=
		N^{-1}(\mathbb Z\setminus N\mathbb Z).
		\]
		A direct calculation gives
		\[
		\mathcal Z(\widehat{\mu}_{N^r,\mathcal D})
		= \bigcup_{k=1}^{\infty} N^{kr}\mathcal Z(\widehat{\delta}_{\mathcal D})
		=	\bigcup_{k=1}^{\infty}
		N^{kr-1}(\mathbb Z\setminus N\mathbb Z)
		\subset N^{r-1}\mathbb Z.
		\]
		If \(t\) is complete, then
		\(t\Lambda_{N^r,N^{r-1}\mathcal D}\) is a spectrum of
		\(\mu_{N^r,\mathcal D}\). Since $0$ and 
		$tN^{r-1}\in
		t\Lambda_{N^r,N^{r-1}\mathcal D},$ 
		orthogonality implies
		\[
		tN^{r-1}\in
		\mathcal Z(\widehat{\mu}_{N^r,\mathcal D})
		\subset N^{r-1}\mathbb Z.
		\]
		Hence \(t\in\mathbb Z\).
		
		It remains to prove that \(\gcd(t,N)=1\). 
		Since \(t\) is complete, 	$t\Lambda_{N^r,N^{r-1}\mathcal D}$ is a spectrum of
		\(\mu_{N^r,\mathcal D}\). Hence, by Theorem~\ref{theo2.1}, $Q_{\Lambda_{N^r,N^{r-1}\mathcal D}}(\xi)\equiv1.$
		Applying the  Ruelle transfer operator in Remark \ref{ruelle}, we obtain
		\[
		1
		=
		\sum_{d\in\mathcal D}
		\left|
		\widehat{\delta}_{\mathcal D}
		\left(
		N^{-r}\bigl(\xi+tN^{r-1}d\bigr)
		\right)
		\right|^2.
		\]
		Letting \(\eta=N^{1-r}\xi\), this becomes
		\[
		1
		=
		\sum_{d\in\mathcal D}
		\left|
		\widehat{\delta}_{\mathcal D}
		\left(
		N^{-1}(\eta+td)
		\right)
		\right|^2
		=
		\sum_{d\in\mathcal D}
		\left|
		\widehat{\delta}_{N^{-1}\mathcal D}
		(\eta+td)
		\right|^2.
		\]
		By Theorem~\ref{theo2.1}, \(t\mathcal D\) is a spectrum of
		\(\delta_{N^{-1}\mathcal D}\).
		
		Suppose that \(\gcd(t,N)=l>1\). Write
		$t=lt', N=lN'$ for some nonzero integer.	Then $N'\in\mathcal D$. Notice that
		\(0,tN'\in t\mathcal D\) and $$tN'=Nt'\notin\Z\setminus N\Z=\mathcal{Z}(\widehat{\delta}_{N^{-1}{\D}}).$$   
		This contradicts $t\D\setminus\{0\}\subset\mathcal{Z}(\widehat{\delta}_{N^{-1}{\D}})$ since $t\D$ is a spectrum of $\delta_{N^{-1}{\D}}$.
	\end{proof}

	\subsection{Integer cycle}
	\ \\
	Integer cycles are the main tool for detecting whether a number coprime to $N$ is complete.
	
	\begin{defi}
		Let $R$ be a positive integer and let $\mathcal C\subset\mathbb Z$. A
		periodic sequence $\{x_k\}_{k=0}^{\infty}\subset\mathbb Z$ is called a
		cycle w.r.t. the pair $(R,\mathcal C)$ if there exists a sequence
		$\{c_k\}_{k=0}^{\infty}\subset\mathcal C$ such that
		\[
		x_{k+1}
		=
		\frac{x_k+c_k}{R},
		\qquad k\geq 0.
		\]
		The points $x_k$ are called cycle points. A sequence
		$\{x_k\}_{k=0}^{\infty}$ is called nonzero if it is not $\{0\}_{k=0}^{\infty}$.
	\end{defi}
	
	
	Integer cycles are closely connected with the complete number problem. This connection was first established by $\L$aba and Wang \cite{LW2002}:
	\begin{theo}{\rm\cite[Theorem 1.3]{LW2002}}\label{theoLW}
		Let $t$ be a positive integer such that $\gcd(t,N)=1$. Then $t$ is incomplete
		for the spectral pair $\left( \mu_{N^r,\mathcal D},
		\Lambda_{N^r,N^{r-1}\mathcal D}\right) $
		if and only if there exists a nonzero integer cycle w.r.t. $
		(N^r,tN^{r-1}\mathcal D)$.
	\end{theo}
	
	\begin{prop}\label{integercycle}
		Let $t$ be a positive integer such that $\gcd(t,N)=1$. Then 
		\begin{enumerate}
			\item[\rm (i)]  {$t$ is incomplete if and only if there exists a nonzero integer cycle w.r.t. $(N^r,t\mathcal D)$;}
			\item[\rm (ii)]  {Suppose that $t$ is incomplete and that $\{x_k\}_{k=0}^{\infty}$ is a nonzero integer cycle w.r.t. $(N^r,t\mathcal D)$ of period $n$. Then, for some
				$\{d_i\}_{i=0}^{\infty}\subset\mathcal D$, 
				\[x_k=t\frac{d_{k-n}+N^rd_{k-n+1}+\cdots+N^{(n-1)r}d_{k-1}}{N^{nr}-1}\in T(N^r,t\mathcal D)\cap\mathbb Z\setminus\{0\},\quad\forall k\ge n;\]}
			\item [\rm (iii)] If $t$ is incomplete, then $kt$ is also incomplete for any $k\in\mathbb N^+$;
			\item [\rm (iv)] If $t$ is incomplete and $\{x_k\}_{k=0}^\infty$ is a nonzero integer cycles w.r.t. $(N^r,t\D)$, then the points within each period are pairwise distinct.
		\end{enumerate}
	\end{prop}
	
	\begin{proof}
		(i) For the necessity, according to Theorem \ref{theoLW}, there exists a nonzero integer cycle $\{x_k\}_{k=0}^{\infty}$ w.r.t. $(N^{r}, tN^{r-1}\D)$. That is, there exists a sequence $\{d_k\}_{k=0}^\infty$ in $\D$ such that
		\[
		x_{k+1}=\frac{x_k+tN^{r-1}d_k}{N^r}\in\Z\setminus\{0\},
		\quad k\in\N.
		\]
		Obviously, $N^{r-1}$ divides $x_k$ for all $k\ge1$. Let $y_k=N^{-(r-1)}x_k$. Then $\{y_k\}_{k=0}^{\infty}$ is a nonzero integer cycle w.r.t. $(N, t\D)$.
		
		For the sufficiency, if $\{x_k\}_{k=0}^{\infty}$ is a nonzero integer cycle w.r.t. $(N, t\D)$, then $\{N^{r-1}x_k\}_{k=0}^{\infty}$ is a nonzero integer cycle w.r.t. $(N^{r}, tN^{r-1}\D)$. By Theorem \ref{theoLW}, $t$ is incomplete.
		
		(ii) By the definition of a cycle,
		\begin{equation}\label{eqxk}
			x_k=\frac{x_{k-1}+td_{k-1}}{N^r}=\frac{td_{k-1}}{N^r}+\frac{td_{k-2}}{N^{2r}}+\cdots+\frac{td_{k-n}}{N^{nr}}+\frac{x_{k-n}}{N^{nr}}.
		\end{equation}
		Since $x_k=x_{k-n}$, we obtain
		\[x_k=t\frac{d_{k-n}+N^rd_{k-n+1}+\cdots+N^{(n-1)r}d_{k-1}}{N^{nr}-1}.\]
		Moreover, iterating \eqref{eqxk} shows that
		$x_k\in T(N^r,t\mathcal D)$. One easily checks that if $x_k=0$, then $d_i=0$ for all $i\in\N$. Hence $\{x_k\}_{k=0}^{\infty}=\{0\}_{k=0}^\infty$, a contradiction.
		
		(iii) If $\gcd(kt,N)>1$, then $kt$ is incomplete by Proposition~\ref{prop2.3}. If $\gcd(kt,N)=1$, then any nonzero integer cycle w.r.t. $(N^r,t\mathcal D)$ yields a nonzero integer cycle w.r.t. $(N^r,kt\mathcal D)$ by multiplying each cycle point by $k$. Hence, Proposition~\ref{integercycle}\,\textup{(i)} implies the result.
		
		(iv) We first claim that two distinct nonzero integer cycles w.r.t. $(N^r,t\mathcal D)$ cannot share a common point. Indeed,let $n$ and $l$ be period of $\{x_k\}_{k=0}^\infty$ and $\{y_k\}_{k=0}^\infty$, respectively. Without loss of generality, we suppose $x_0=y_0$. Similar as the proof of (ii), there exist $\{c_k\}_{k=0}^\infty$ and $\{d_k\}_{k=0}^\infty$ in $\D$ such that
		\begin{eqnarray*}
			x_0=t\frac{c_0+N^rc_1+\cdots+N^{(n-1)r}c_{n-1}+\cdots+N^{(nl-1)r}c_{nl-1}}{N^{nlr}-1}
		\end{eqnarray*}
		and
		\begin{eqnarray*}
			y_0=t\frac{d_0+N^rd_1+\cdots+N^{(n-1)r}d_{n-1}+\cdots+N^{(nl-1)r}d_{nl-1}}{N^{nlr}-1}.
		\end{eqnarray*}
		$x_0=y_0$ forces $c_i=d_i$ for $0\le i\le nl-1$. Then $x_i=y_i$ for all $i\in\N$, which leads a contradiction. 
		
		Now we prove that the points within each period are pairwise distinct. Without loss of generality, suppose that $x_0=x_m$ for some $0<m<n$. Then, by repeating $x_1,\ldots,x_m$ and $x_m,\ldots,x_n$ periodically, respectively, we obtain two distinct integer cycles w.r.t. $(N^r,t\mathcal D)$. The two integer cycles share the common point $x_m$, which contradicts the claim.
	\end{proof}

	
	\begin{coro}\label{corolenr-1}
		Let \(t\) be a positive integer such that \(\gcd(t,N)=1\) and \(1\le t<\frac{N^r-1}{N-1}\). Then \(t\) is complete. However, \(\frac{N^r-1}{N-1}\) is incomplete.
	\end{coro}
	\begin{proof}
		Suppose that $1\le t<\frac{N^r-1}{N-1}$ is incomplete. By Proposition~\ref{integercycle} (i), there exists a nonzero integer cycle w.r.t. $(N^r,t\mathcal D)$. Note that
		\[T(N^r,t\mathcal D)\subset\left[0,\frac{t(N-1)}{N^r-1}\right]\subset [0,1).\]
		Thus
		\[
		T(N^r,t\mathcal D)\cap\mathbb Z\setminus\{0\}
		=
		\emptyset,
		\]
		which contradicts Proposition~\ref{integercycle} (ii).
		
		It is enough to observe that
		$$1	=\frac{1+\frac{N^r-1}{N-1}(N-1)}{N^r}.$$
		Thus $\{1\}_{k=0}^{\infty}$ is a nonzero cycle w.r.t. $\left( N^r,\frac{N^r-1}{N-1}\mathcal D\right) $. $\frac{N^r-1}{N-1}$ is incomplete follows from Proposition~\ref{integercycle} (i).
	\end{proof}
	
	
	\subsection{Primitive Numbers}\ \\
	Primitive numbers are a class of incomplete integer. Such numbers will play an important role in the proof of our main theorem. 
	
	\begin{defi}
		Let $t$ be a positive integer such that $\gcd(t,N)=1$. We say that $t$ is
		\emph{primitive} if $t$ is incomplete and any proper divisor of $t$ is complete.
	\end{defi}
	
	To describe the properties of primitive numbers, we first recall the definition of the order of $a$ modulo $t$.
	\begin{defi}\label{defi4.1}
		Let $a$ and $t$ be positive integers such that $\gcd(a,t)=1$.
		\begin{enumerate}
			\item[(i)] The order of $a$ modulo $t$ is defined by
			$$O_a(t)=\min\{k\in\mathbb Z:a^{k}\equiv1\pmod t\},$$
			which equals the order of the cyclic group $G_{a,t}:=\{a^k\pmod t:1\le k\le O_a(t)\}$.
			\item[(ii)] We define
			\[\ell_a(t)=\max\left\{k\in\mathbb N^+: O_a(t^k)=O_a(t)\right\}.\]
		\end{enumerate}
	\end{defi}
	\begin{prop}\label{prop-primitive-properties}
		Let $t$ be a positive integer such that $\gcd(t,N)=1$. Then the following hold.
		\begin{enumerate}
			\item[\textup{(i)}]
			The number $t$ is incomplete if and only if $t$ is divisible by a primitive number.
			
			\item[\textup{(ii)}]
			If $t$ is primitive, then for any nonzero cycle
			$\{x_k\}_{k=0}^{\infty}$ w.r.t. $(N^r,t\mathcal D)$, one has
			\[
			\gcd(t,x_k)=1
			\qquad
			\text{for all } k\geq 0.
			\]
			
			\item[\textup{(iii)}]
			If $t$ is primitive, then the minimal period of any nonzero integer cycle w.r.t. $(N^r,t\mathcal D)$ equals $O_{N^r}(t)$.
			
			\item[\textup{(iv)}]
			If $t$ is primitive, then there exist
			$d_0,d_1,\ldots,d_{O_{N^r}(t)-1}\in\mathcal D$, not all equal to $0$, such that
			\[t\frac{d_0+N^rd_1+\cdots+N^{r(O_{N^r}(t)-1)}d_{O_{N^r}(t)-1}}{N^{rO_{N^r}(t)}-1}\in\mathbb Z\setminus\{0\}.\]
		\end{enumerate}
	\end{prop}

	\begin{proof}
		\textup{(i)}
		Suppose first that $t$ is incomplete. If $t$ is primitive,
		there is nothing to prove. Otherwise, $t$ has a proper divisor $t_1$ which is incomplete. If $t_1$ is primitive, we are done.
		If not, we repeat the same argument. Since divisors decrease at each step, this process must terminate. Hence $t$ is divisible by a primitive number.
		
		Conversely, suppose that $t$ is divisible by a primitive number $s$. Since $s$ is incomplete, Proposition~\ref{integercycle} (iii) implies that any
		positive multiple of $s$, in particular $t$, is incomplete.
		
		\medskip
		
		\textup{(ii)}
		Let $\{x_k\}_{k=0}^{\infty}$ be a nonzero cycle w.r.t.
		$(N^r,t\mathcal D)$. Suppose, to the contrary, that
		\[
		\gcd(t,x_0)=m>1.
		\]
		Write $t=mt'$ and $x_0=mx'_0$. If the cycle has period $n$, then
		\[
		x_0=x_n=\frac{x_{n-1}+td_{n-1}}{N^r}
		\]
		for some $d_{n-1}\in\mathcal D$. Hence
		$x_{n-1}=N^rx_0-td_{n-1}$
		is divisible by $m$. Applying the same argument backwards along the cycle, we conclude that
		$m\mid x_k$ for any $k\in\N$.
		Therefore $
		\left\{\frac{1}{m}x_k\right\}_{k=0}^{\infty}$ is a nonzero integer cycle w.r.t. $(N^r,t'\mathcal D)$. By
		Proposition~\ref{integercycle} (i), $t'$ is incomplete. This contradicts the fact that $t$ is primitive, since $t'$ is a positive proper divisor of $t$.
		Thus $\gcd(t,x_0)=1$. The same argument applies to any cycle point $x_k$, and the proof of {(ii)} is complete.
		
		\medskip
		
		\textup{(iii)}
		Let $n$ be the minimal period of the cycle.
		By Proposition~\ref{integercycle} (ii), any cycle point satisfies
		\[x_k\in T(N^r,t\mathcal D)\subset\left[0,\frac{t(N-1)}{N^r-1}\right]\subset[0,t).\]
		From the definition of integer cycle, $N^rx_{k+1}=x_k+td_k$, then we get $N^rx_{k+1}\equiv x_k \pmod t$. Repeating this argument,
		Thus
		\[x_{O_{N^r}(t)}\equiv ( N^r)^ {O_{N^r}(t)}x_{O_{N^r}(t)}\equiv ( N^r)^ {O_{N^r}(t)-1}x_{O_{N^r}(t)-1} \equiv\cdots\equiv N^rx_1\equiv x_0 \pmod t,\]
		Since both $x_{O_{N^r}(t)}$ and $x_0$ lie in
		$[0,t)$, we have $x_{O_{N^r}(t)}=x_0$. Proposition~\ref{integercycle} (iv) states that the points within each period are pairwise distinct. Hence $n\mid {O_{N^r}(t)}$. 
		
		If $n<O_{N^r}(t)$, then $x_n=x_0$ gives $N^{rn}x_0\equiv x_0 \pmod t$. By \textup{(ii)}, $\gcd(t,x_0)=1$. Hence
		\[(N^r)^n\equiv 1\pmod t,\]
		which contradicts the definition of $O_{N^r}(t)$. Therefore $n=O_{N^r}(t)$.
		
		\textup{(iv)}
		The conclusion follows directly from (iii) and Proposition~\ref{integercycle} (ii).
	\end{proof}
	
	Finally, we present some additional properties of the order.	
	\begin{prop}\label{proporder}
		Let $a,t\ge 2$ be integers with $\gcd(a,t)=1$. Then the following hold.
		
		\begin{enumerate}
			\item[\textup{(i)}]
			For every $n\in\mathbb N^+$,
			$t\mid a^n-1$ if and only if 
			$O_a(t)\mid n.$
			
			\item[\textup{(ii)}]
			For every $r\in\mathbb N^+$,
			$O_{a^r}(t)			=
			\frac{O_a(t)}{\gcd(O_a(t),r)}.$
		\end{enumerate}
		
		Furthermore, let $p$ be prime with $p\nmid a$, then the following hold.
		
		\begin{enumerate}
			\item[\textup{(iii)}]
			$O_a(p)\mid p-1.$
			In particular, $\gcd(p,O_a(p))=1.$
			
			\item[\textup{(iv)}]
			For every $k\in\mathbb N^+$,
			$O_a(p^k)
			=
			O_a(p)\,p^{\max\{k-\ell_a(p),0\}}.
			$
			
			\item[\textup{(v)}]
			If $r=p^j s$, where $j\ge 0$ and $\gcd(s,p)=1$, then $\ell_{a^r}(p)=\ell_a(p)+j.$
			In particular, if $p\nmid r$, then
			$\ell_{a^r}(p)=\ell_a(p).$
		\end{enumerate}
	\end{prop}
	
	\begin{proof}
		\textup{(i)} follows directly from the definition of $O_a(t)$.
		
		\textup{(ii)}
		Let $u \in \mathbb{N^+}$. Then
		\[
		(a^r)^u\equiv 1\pmod t
		\quad\Longleftrightarrow\quad
		O_a(t)\mid ru.
		\]
		Let $d=\gcd(O_a(t),r)$, then the smallest positive integer $u$ satisfying $O_a(t)\mid ru$ is $\frac{O_a(t)}{d}$. Hence
		\[
		O_{a^r}(t)=\frac{O_a(t)}{d}
		=
		\frac{O_a(t)}{\gcd(O_a(t),r)}.
		\]
		
		\textup{(iii)} follows from Fermat's little theorem, since $	a^{p-1}\equiv 1\pmod p.$
		Thus $O_a(p)\mid p-1$, and hence $\gcd(p,O_a(p))=1$.
		
		\textup{(iv)} is proved in \cite[Proposition~2.21]{DK18}.
		
		\textup{(v)}
		Write $r=p^j s, \gcd(s,p)=1.$
		By \textup{(iii)}, $\gcd(p,O_a(p))=1$. Using \textup{(ii)} and \textup{(iv)}, we get
		\[
		\begin{aligned}
			O_{a^r}\bigl(p^{\ell_a(p)+j}\bigr)
			=
			\frac{
				O_a\bigl(p^{\ell_a(p)+j}\bigr)
			}{
				\gcd\bigl(O_a(p^{\ell_a(p)+j}),r\bigr)
			}                                      
			=
			\frac{
				p^jO_a(p)
			}{
				\gcd(p^jO_a(p),p^js)
			}
			=
			\frac{O_a(p)}{\gcd(O_a(p),s)}
			=
			O_{a^r}(p).
		\end{aligned}
		\]
		Similarly,
		\[
		\begin{aligned}
			O_{a^r}\bigl(p^{\ell_a(p)+j+1}\bigr)
			=
			\frac{
				O_a\bigl(p^{\ell_a(p)+j+1}\bigr)
			}{
				\gcd\bigl(O_a(p^{\ell_a(p)+j+1}),r\bigr)
			}                                      
			=
			\frac{
				p^{j+1}O_a(p)
			}{
				\gcd(p^{j+1}O_a(p),p^js)
			}
			=
			p\,\frac{O_a(p)}{\gcd(O_a(p),s)}
			=
			pO_{a^r}(p).
		\end{aligned}
		\]
		Therefore
		\[
		\ell_{a^r}(p)=\ell_a(p)+j
		\]
		by (iv).
	\end{proof}

	\section{Proof of Theorem \ref{theomain}  }\label{secproof}
	
	In this section, we prove the Theorem \ref{theomain}. First, We consider Theorem \ref{theomain} (2), namely $\log_p\frac{N^r-1}{N-1}\in\mathbb N$.  The conclusion is then easy to verify.
	\begin{proof}[Proof of Theorem \ref{theomain} (2)]
		Since $\log_p \frac{N^r-1}{N-1}\in\mathbb N$, we have $p^m=1+N+\cdots+N^{r-1}$ for some $m\in\N^+$. Then $\gcd(p,N)=1$. 
		If $n<\log_p \frac{N^r-1}{N-1}$, then $p^n<\frac{N^r-1}{N-1}.$
		By Corollary~\ref{corolenr-1}, \(p^n\) is complete.
		If \(n\ge \log_p \frac{N^r-1}{N-1}\), then $\frac{N^r-1}{N-1}\mid p^n.$
		By Proposition~\ref{prop-primitive-properties}, \(p^n\) is incomplete.
	\end{proof}
	Hence, we only need to prove Theorem~\ref{theomain}(1). To prove this, we use a result of Zhou, Xing, and Ai  \cite[Lemma~5.3]{ZXA2026}, together with Proposition~\ref{plnrp}. This reduces the problem for all powers $p^n$ to proving that the critical power $p^{\ell_{N^r}(p)}$ is complete.

	\begin{lem}\label{Ai}
		Let $s,t>1$ be odd integers. If $
		O_{N^r}(st)
		\geq
		\frac{s(N-1)}{N^r-1}
		+
		\frac{N^{r-1}-1}{N^{r-2}},$
		then $st$ is not a primitive number.
	\end{lem}
	\begin{prop}\label{plnrp}
		Let $p>N$ be a prime. If $p^{\ell_{N^r}(p)}$ is complete, then
		$p^n$ is complete for any $n\in\mathbb N^+$.
	\end{prop}
	
	\begin{proof}
		Suppose that $p^n$ is incomplete for some
		$n\in\mathbb N^+$. By Proposition \ref{prop-primitive-properties}, there
		exists $j\le n$ such that $p^{j}$ is primitive. Since $p^{\ell_{N^r}(p)}$ is complete, Proposition~\ref{integercycle} (iii) implies that $j>\ell_{N^r}(p)$. Write $j=\ell_{N^r}(p)+k$, $k\in\mathbb N^+$, and set $s=p^k, t=p^{\ell_{N^r}(p)}.$
		Then $st=p^{\ell_{N^r}(p)+k}$ is primitive. By
		Proposition \ref{proporder} \textup{(iv)},
		$\frac{O_{N^r}(st)}{O_{N^r}(t)}=p^k.$
		Since $p^k>N$,
		\[\begin{aligned}
			p^k=\frac{p^k(N-1)}{N^r-1}+\frac{p^k(N^r-N)}{N^r-1} >\frac{s(N-1)}{N^r-1}+\frac{N^2(N^{r-1}-1)}{N^r-1}>\frac{s(N-1)}{N^r-1}+\frac{N^{r-1}-1}{N^{r-2}}.
		\end{aligned}\]
		Consequently,
		\[O_{N^r}(st)\ge p^k>\frac{s(N-1)}{N^r-1}+\frac{N^{r-1}-1}{N^{r-2}}.\]
		By Lemma \ref{Ai}, $st$ is not primitive, a
		contradiction. Hence $p^n$ is complete for any $n\in\mathbb N^+$.
	\end{proof}

	Thus it remains to prove the
	following theorem.
	
	\begin{theo}
		Let $p>N$ be a prime such that 
		$$\log_p\frac{N^r-1}{N-1}\notin\N.$$
		Then $p^{\ell_{N^r}(p)}$ is complete for any $n\in\N$.
	\end{theo}
	
	We prove this by considering two cases:
	\begin{enumerate}
		\item [(i)] one of $O_{N^r}(p)$ and $r$ is even;
		\item [(ii)] both $O_{N^r}(p)$ and $r$ are odd.
	\end{enumerate}
	
	\subsection{Proof of Case (i)}
	\begin{lem}\label{lem4.1}
		Let $t$ be a positive integer such that $\gcd(t,N)=1$. If $t$ is incomplete and $\{x_k\}_{k=0}^{\infty}$ is an integer cycle w.r.t. $(N^r,t\D)$. Then for any $i\in\N$, $N^{ir}x_0\pmod t$ is an integer cycle point.
	\end{lem}
	\begin{proof}
		We suppose $n$ is the minimal period of $\{x_k\}_{k=0}^\infty$. By the definition of cycle,
		$$N^{r}x_k\equiv x_{k-1}\pmod{t},\quad k\in \N^+.$$
		Then for any $i\in \N^+$, 
		$$N^{ir}x_0\equiv N^{(i-1)r}x_{n-1}\equiv\cdots\equiv x_{(n-i)\ (\text{mod }n)}\pmod{t}.$$
		This completes the proof.
	\end{proof}
	\begin{prop}\label{theocase1}
		Let $p>N$ be a prime such that
		$$\log_p{\frac{N^r-1}{N-1}}\notin\N.$$
		If one of $O_{N^r}(p)$ and $r$ is even, then $p^{\ell_{N^r}(p)}$ is complete.
	\end{prop}
	
	\begin{proof}
		By Proposition \ref{proporder} \textup{(iv)}, $$O_{N^r}(p^{\ell_{N^r}(p)})=O_{N^r}(p).$$
		Suppose that $p^{\ell_{N^r}(p)}$ is incomplete and $\{x_k\}_{k=0}^\infty$ is a nonzero integer cycle w.r.t. $(N^r,p^{\ell_{N^r}(p)}\D)$. 
		The hypothesis on the prime $p$ implies that $\frac{N-1}{N^r-1}p^{\ell_{N^r}(p)}\notin\Z$. In fact, If $\frac{N^r-1}{N-1}\mid p^{\ell_{N^r}(p)}$, then $p^{j_0}=\frac{N^r-1}{N-1}$ for some positive integer $j_0\le\ell_{N^r}(p)$. Therefore, $j_0=\log_p\frac{N^r-1}{N-1}\in\N$. A contradiction. Then Proposition~\ref{integercycle} (ii) implies that 
		\begin{equation*}
			x_k\in\left(0,\frac{N-1}{N^r-1}p^{\ell_{N^r}(p)} \right),\quad\forall k\in\N.
		\end{equation*}
		Without loss of generality, we let $x_0=\max_{k\in\N}x_k$.
		
		Consider the first case in which $O_{N^r}(p)$ is even. Then
		$$\left( N^{r\frac{O_{N^r}(p)}{2}}-1\right) \left( N^{r\frac{O_{N^r}(p)}{2}}+1\right) = N^{r{O_{N^r}(p)}}-1\equiv 0\pmod{p^{\ell_{N^r}(p)}}.$$
		By the definition of $O_{N^r}(p^{\ell_{N^r}(p)})$, the only solution of this equation is $N^{r\frac{O_{N^r}(p)}{2}}\equiv -1\pmod{p^{\ell_{N^r}(p)}}$. 
		According to Lemma \ref{lem4.1},
		$$N^{r\frac{O_{N^r}(p)}{2}}x_0\pmod{p^{\ell_{N^r}(p)}}=-x_0\pmod{p^{\ell_{N^r}(p)}}$$
		is also an integer cycle point. But
		$$-x_0+p^{\ell_{N^r}(p)}\ge -\frac{N-1}{N^r-1}p^{\ell_{N^r}(p)}+p^{\ell_{N^r}(p)}>\frac{N-1}{N^r-1}p^{\ell_{N^r}(p)}.$$
		This implies the integer cycle point $$-x_0\pmod{p^{\ell_{N^r}(p)}}>\frac{N-1}{N^r-1}p^{\ell_{N^r}(p)}.$$
		Hence, this contradicts the conclusion of Proposition~\ref{integercycle} (ii) that $x_k\in\left(0,\frac{N^r-1}{N-1}p^{\ell_{N^r}(p)}\right]$.
		
		It remains to consider the case where $O_{N^r}(p)$ is odd and $r$ is even. Then
		$$\left( N^{r\frac{O_{N^r}(p)+1}{2}}+N^{\frac{r}{2}}\right) \left( N^{r\frac{O_{N^r}(p)+1}{2}}-N^{\frac{r}{2}}\right)= N^{r(O_{N^r}(p)+1)}-N^r\equiv0\pmod{p^{\ell_{N^r}(p)}}.$$
		Noticing that $p\nmid 2N^{\frac{r}{2}}$ since $p>N$, we have
		$$p\nmid\gcd\left( N^{r\frac{O_{N^r}(p)+1}{2}}-N^{\frac{r}{2}},N^{r\frac{O_{N^r}(p)+1}{2}}+N^{\frac{r}{2}}\right).$$
		Therefore one of the two factors must be divisible by $p^{\ell_{N^r}(p)}$, namely,
		$$N^{r\frac{O_{N^r}(p)+1}{2}}\equiv -N^{\frac{r}{2}}\pmod{p^{\ell_{N^r}(p)}}\quad\text{or}\quad N^{r\frac{O_{N^r}(p)+1}{2}}\equiv N^{\frac{r}{2}}\pmod{p^{\ell_{N^r}(p)}}.$$
		If $N^{r\frac{O_{N^r}(p)+1}{2}}\equiv -N^{\frac{r}{2}}\pmod{p^{\ell_{N^r}(p)}}$, then 
		$$N^{r\frac{O_{N^r}(p)+1}{2}}x_0\pmod{p^{\ell_{N^r}(p)}}=-N^{\frac{r}{2}}x_0\pmod{p^{\ell_{N^r}(p)}}$$
		is also an integer cycle point by Lemma \ref{lem4.1}. Therefore,
		\begin{eqnarray*}
			-N^{\frac{r}{2}}x_0+p^{\ell_{N^r}(p)}&>&-N^{\frac{r}{2}}\cdot\frac{N-1}{N^r-1}p^{\ell_{N^r}(p)}+p^{\ell_{N^r}(p)}\\
			&=&\left(\frac{N^r-1}{N-1}-N^{\frac{r}{2}}\right)\cdot \frac{N-1}{N^r-1}p^{\ell_{N^r}(p)}\ge\frac{N-1}{N^r-1}p^{\ell_{N^r}(p)}.
		\end{eqnarray*}
		Consequently, the integer cycle point $$-N^{\frac{r}{2}}x_0\pmod{p^{\ell_{N^r}(p)}}>\frac{N-1}{N^r-1}p^{\ell_{N^r}(p)}.$$ This contradicts Proposition~\ref{integercycle} (ii). If $N^{r\frac{O_{N^r}(p)+1}{2}}\equiv N^{\frac{r}{2}}\pmod{p^{\ell_{N^r}(p)}}$, then 
		$$N^{r\frac{O_{N^r}(p)+1}{2}}x_0\pmod{p^{\ell_{N^r}(p)}}=N^{\frac{r}{2}}x_0\pmod{p^{\ell_{N^r}(p)}}$$
		is also an integer cycle point. Since
		\begin{eqnarray*}
			N^{\frac{r}{2}}x_0-p^{\ell_{N^r}(p)}< N^{\frac{r}{2}}\cdot\frac{N-1}{N^r-1}p^{\ell_{N^r}(p)}-p^{\ell_{N^r}(p)}\le 0,
		\end{eqnarray*}
		it follows that the integer cycle point $$N^{\frac{r}{2}}x_0\pmod{p^{\ell_{N^r}(p)}}=N^{\frac{r}{2}}x_0>x_0.$$
		This contradicts the maximality of $x_0$. Hence complete the proof.
	\end{proof}

	\subsection{Proof of Case (ii)}
	This is the trickiest part of the proof, so we need to further analyze the relation between $O_N(p)$ and $r$. By Proposition \ref{proporder} \textup{(ii)}, we have
	\[O_N(p)=O_{N^r}(p)\gcd(O_N(p),r)\]
	is odd. Hence we further divide this case into the following three subcases:
	\[r\nmid O_N(p),\quad r=O_N(p),\quad r\mid O_N(p)\ \text{and}\ r<O_N(p).\]
	\begin{prop}\label{theomain2}
		Let $p$ be a prime such that $\gcd(p,N)=1$. If $r$ and $O_N(p)$ are odd and $r\nmid O_N(p)$, then $p^{\ell_{N^r}(p)}$ is complete.
	\end{prop}
	\begin{proof}
		We write $r=p^ks$ for some $k\geq 0,s\in\N$ and $p\nmid s$. Then we have $\ell_{N^r}(p)=\ell_N(p)+k$ by Proposition \ref{proporder} (v).
		
		Case 1: $k\ge 1$. The function $f(x)=\frac{N^x-1}{x}$ is increasing on $[3,\infty)$ when $N\ge2$. Then for any odd prime $p$, by AM-GM inequality, we always have that
		\begin{eqnarray}\label{eq6.5}
			N^{\frac{p-1}{2}}=\Bigg(\prod_{i=0}^{p-1}N^i \Bigg)^{\frac{1}{p}}\le\frac{1}{p}\sum_{i=0}^{p-1}N^i=\frac{f(p)}{N-1}\le \frac{f(p^k)}{N-1}=\frac{N^{p^k}-1}{p^k(N-1)}\le \frac{N^{p^{ks}}-1}{p^k(N-1)}.
		\end{eqnarray}
		Fermat's little theorem shows that $O_N(p)\mid p-1$. Noticing that $O_N(p)$ and $p$ are odd numbers, then $O_N(p)\le\frac{p-1}{2}$.
		By the definition of $\ell_{N}(p)$, we conclude that $p^{\ell_{N}(p)}\mid N^{O_N(p)}-1$. According to \eqref{eq6.5},
		$$p^{\ell_{N^r}(p)}= p^k\cdot p^{\ell_{N}(p)}\le p^k\cdot \left( N^{O_N(p)}-1\right)\le p^k\cdot \left( N^{\frac{p-1}{2}}-1\right)< p^k\cdot \frac{N^{p^ks}-1}{p^k(N-1)}=\frac{N^r-1}{N-1}.$$
		Corollary \ref{corolenr-1} implies that $p^{\ell_{N^r}(p)}$ is a complete number.
		
		Case 2: $k=0$ and $r\nmid O_N(p)$. We have
		$$O_N(p^{\ell_{N^r}(p)})=O_N(p^{\ell_{N}(p)})=O_N(p)=\beta r+\alpha$$
		for some $\alpha,\beta\in\N$ and $1\le\alpha\le r-1$. Then $$N^{r(\beta+1)}\equiv N^{r-\alpha}\pmod {p^{\ell_{N^r}(p)}}.$$
		Suppose that $p^{\ell_{N^r}(p)}$ is incomplete and $\{x_k\}_{k=0}^\infty$ is a nonzero integer cycle w.r.t. $(N^r,p^{\ell_{N^r}(p)}\D)$. We let $x_0=\max_{k\in\N}x_k$. Then 
		$$N^{r(\beta+1)}x_0\pmod{p^{\ell_{N^r}(p)}}=N^{r-\alpha}x_0\pmod{p^{\ell_{N^r}(p)}}$$
		is also an integer cycle point according to Lemma \ref{lem4.1}. Since
		\begin{eqnarray*}
			N^{r-\alpha}x_0-p^{\ell_{N^r}(p)}&<& N^{r-\alpha}\cdot\frac{N-1}{N^r-1}p^{\ell_{N^r}(p)}-p^{\ell_{N^r}(p)}\\
			&=&-\left(1+N+\cdots+N^{r-1}-N^{r-\alpha}\right)\frac{N-1}{N^r-1}p^{\ell_{N^r}(p)},
		\end{eqnarray*}
		it follows that the integer cycle point $$N^{r(\beta+1)}x_0\pmod{p^{\ell_{N^r}(p)}}=N^{r-\alpha}x_0>x_0.$$ 
		This contradicts the maximality of $x_0$. Hence we reach the conclusion.
	\end{proof}
	
	\begin{prop}\label{theomain3}
		Let $p>N$ be a prime and $\log_p\frac{N^r-1}{N-1}\notin\N$. If $O_N(p)=r$, then $p^n$ is complete for any $n\in\N$.
	\end{prop}
	\begin{proof}
		Since $r=O_N(p)$, $\gcd(r,p)=1$ by Fermat's little Theorem. Then Proposition \ref{proporder} (v) indicates $\ell_{N^r}(p)=\ell_{N}(p)$. Hence $p^{\ell_{N^r}(p)}=p^{\ell_{N}(p)}$ divides $N^r-1$. The prime $p>N$ implies that $\gcd(p,N-1)=1$. Then $p^{\ell_{N^r}(p)}\mid\frac{N^r-1}{N-1}$. According to the condition, $p^j\neq\frac{N^r-1}{N-1}$ for any $j\in\N$. Then $p^{\ell_{N^r}(p)}<\frac{N^r-1}{N-1}$. Therefore, Corollary \ref{corolenr-1} implies that $p^{\ell_{N^r}(p)}$ is complete.
	\end{proof}
	\begin{prop}\label{theomain4}
		Let $p>N$ be a prime such that
		$$\log_p\frac{N^r-1}{N-1}\notin\N.$$
		If $r$ is a proper divisor of $O_N(p)$, then $p^{\ell_{N^r}(p)}$ is complete.
	\end{prop}
	
	\begin{proof}
		Set $ \beta:=O_{N^r}(p)=\frac{O_N(p)}{r}>1$. By Proposition \ref{proporder} (ii), we have
		$$O_{N^r}(p^{\ell_{N^r}(p)}) = O_{N^r}(p) = \beta.$$
		It then follows that $p^{\ell_{N^r}(p)} \mid N^{\beta r}-1$.
		Since $\beta<O_N(p)$, the definition of $O_N(p)$ gives $p\nmid N^\beta-1.$
		Suppose to the contrary that $p^{\ell_{N^r}(p)}$ is incomplete. By
		Proposition \ref{prop-primitive-properties}, there exists an integer $1\le j_0\le{\ell_{N^r}(p)}$ such that $p^{j_0}$ is primitive. Since
		$j_0\le{\ell_{N^r}(p)}$, $O_{N^r}(p^{j_0})=O_{N^r}(p)=\beta$.
		Proposition \ref{prop-primitive-properties} therefore yields digits $d_0,\ldots,d_{\beta-1}\in\D$,
		not all zero, such that
		\begin{equation}\label{eq:cycle-beta}
			p^{j_0} \cdot \frac{\sum_{i=0}^{\beta-1} d_i N^{ir}}{N^{\beta r} - 1} \in \mathbb{Z}.
		\end{equation}
		
		We let $M:=\frac{N^{\beta r}-1}{N^\beta-1}.$
		As $\beta<O_N(p)$, it follows that $p \nmid N^\beta - 1$. Since $p$ is a prime, we have $\gcd(p^{j_0}, N^\beta - 1) = 1$. Because
		$$\frac{(N^\beta - 1)M}{p^{j_0}} = \frac{N^{\beta r} - 1}{p^{j_0}} \in \mathbb{Z},$$
		we have $p^{j_0} \mid M$. From \eqref{eq:cycle-beta}, we obtain
		\begin{eqnarray}\label{eq:reduced-beta}
			\frac{\sum_{i=0}^{\beta-1} d_i N^{ir}}{N^\beta - 1} \in \mathbb{Z}.
		\end{eqnarray}
		For $0 \le i,j \le \beta-1$, we assume without loss of generality that $ir > j$. It is easy to see that $N^{ir} \equiv N^j \pmod{N^\beta - 1}$ if $\beta \mid ir - j$. Therefore, $N^{ir} \pmod{N^\beta - 1} = N^{ir \pmod{\beta}}$.
		
		Consider the map $\varphi=rx\pmod\beta$ from $\mathbb{Z}_\beta$ to $\mathbb{Z}_\beta$. Let $d=\gcd(\beta,r)$ and denote $\beta'=\frac{\beta}{d}$.
		Writing $\beta=d\beta'$ and $r=dr'$, we have
		$\gcd(r',\beta')=1$. Then $\varphi$ is a $d$-to-$1$ map and
		$${\rm Im}(\varphi)=d\left\{0,1,\ldots,\beta'-1\right\}.$$
		That is, for any $1\le j\le \beta'-1$, $\varphi^{-1}(dj)$ has $d$ elements. Then we have
		$$\sum_{i=0}^{\beta-1} d_i N^{ir} \equiv \sum_{j=0}^{\beta'-1} \sum_{k\in\varphi^{-1}(dj)} d_{k} N^{dj} \pmod{N^\beta - 1}.$$
		Together with \eqref{eq:reduced-beta}, this gives \begin{equation*}
			\frac{\sum_{j=0}^{\beta'-1} \sum_{k\in\varphi^{-1}(dj)} d_{k} N^{dj}}{N^\beta-1}\in\mathbb{Z}.
		\end{equation*}
		However,
		$$\sum_{j=0}^{\beta'-1} \sum_{k\in\varphi^{-1}(dj)} d_{k} N^{dj}\le d\left( N-1\right) \sum_{j=0}^{\beta'-1}N^{dj}=\frac{d(N-1)}{N^d-1}\left( N^\beta-1\right).$$
		This forces that the equality holds i.e. $d=1$ and $d_i=N-1$ for all $0\le i\le \beta-1$. It then follows from \eqref{eq:cycle-beta} that 
		$$\frac{p^{j_0}(N-1)}{N^r-1} \in \mathbb{Z}.$$
		Then we have $\frac{N^r-1}{N-1}\mid p^{j_0}$. Since $p$ is prime, we can further deduce that $p^{k_0}=\frac{N^r-1}{N-1}$ for some integer $1\le k_0\le j_0$. Thus
		$\log_p\frac{N^r-1}{N-1}=k_0\in\N$. A contradiction.
		
		Hence $p^{\ell_{N^r}(p)}$ is complete.
	\end{proof}

	Combining Theorem \ref{theomain} and Corollary \ref{corolenr-1}, we have that
	\begin{coro}
		Let $p\nmid N$ be a prime and $p\neq\frac{N^r-1}{N-1}$. Then $p$ is complete.
	\end{coro}

	\section{Applications and examples }\label{sec5}
	
	We first prove Theorem~\ref{theCHW}. We then examine the three known
	nontrivial positive integer solutions of the Nagell--Ljunggren equation
	from the viewpoint of the complete number problem. Finally, we give two
	examples with \(b\neq N^r\), showing that even a prime numbers may fail to
	be complete.

	\begin{proof}[Proof of Theorem \ref{theCHW}]	
		Suppose that $\log_p(2^r-1)\in\mathbb N.$
		Then $2^r-1=p^m$
		for some \(m\in\mathbb N^+\). If \(m\ge2\), then $2^r-p^m=1,$
		which is impossible by the Catalan--Mih\u{a}ilescu theorem
		\cite{M2004}. Thus \(m=1\), and consequently $p=2^r-1.$
		
		Therefore, if \(p\neq2^r-1\), then $\log_p(2^r-1)\notin\mathbb N.$ Since \(p\) is an odd prime, \(p>2\). By Theorem~\ref{theomain} \textup{(1)}, for the spectral pair $	\left(\mu_{2^r,\{0,1\}},\Lambda_{2^r,\{0,2^{r-1}\}}\right)$, $p^n$ is complete for any \(n\in\mathbb N^+\).
	\end{proof}

	We next consider the three known nontrivial positive integer solutions
	of the Nagell--Ljunggren equation:
	\[
	\textup{(i)}~\frac{3^5-1}{3-1}=11^2,\qquad
	\textup{(ii)}~\frac{7^4-1}{7-1}=20^2,\qquad
	\textup{(iii)}~\frac{18^3-1}{18-1}=7^3.
	\]
	
	\begin{exam}\label{exam:Nagell-Ljunggren}
		We consider the solutions of Nagell--Ljunggren equation (i) and (iii).
		\begin{enumerate}
			\item[\textup{(1)}]
			By the solution {\rm (i)}, 
			$$\log_{11}\frac{3^5-1}{3-1}=2.$$
			Hence, $11$ is complete, whereas $11^n$ is incomplete for any $n\ge2$ for the spectral pair
			$\left(\mu_{3^5,{0,1,2}},\Lambda_{3^5,3^4{0,1,2}}\right)$.
			\item[\textup{(2)}]
			By the solution {\rm (iii)}, 
			$$\log_{7}\frac{18^3-1}{18-1}=3.$$
			Hence, $7$ and $7^2$ are complete, whereas $7^n$ is incomplete for any \(n\ge3\) for the spectral pair $\left(
			\mu_{3^5,\{0,1,2\}},
			\Lambda_{3^5,3^4\{0,1,2\}}
			\right)$.
		\end{enumerate}
	\end{exam}
	
	We next illustrate the role of the restriction \(p>N\) in Theorem~\ref{theomain}.
	\begin{exam} \label{exam527} 
		For the spectral pair $\left(\mu_{5^2,\{0,1,2,3,4\}},\Lambda_{5^2,5\{0,1,2,3,4\}}\right)$. \(2\) and \(2^2\) are complete by
		Corollary~\ref{corolenr-1}. $2^3$ is incomplete since $1=\frac{1+2^3\cdot3}{25},$
		and \(\{1\}_{k=0}^{\infty}\) is an integer cycle w.r.t.
		\((25,2^3\{0,1,\ldots,4\})\).
	\end{exam}

	The following two examples show that, beyond the case \(b=N^r\), even
	prime numbers can fail to be complete for spectral pairs $	\left(\mu_{b,q\mathcal D},\Lambda_{b,\mathcal D}\right).$
	
	\begin{exam}\label{exam2.12}
		For the case $b=6$, $N=3$ and let $p=760891$. Then $p$ is a prime number and there exists a non-zero integer cycle w.r.t. $(6,p\{0,1,2\})$. We list the terms of a period in order:
		\begin{equation}\nonumber
			\begin{split}
				802,253764,42294,&7049,127990,274962,45827,\\&134453,149224,278501,173232,28872,4812.
			\end{split}
		\end{equation}
		Then the prime number $p$ is an incomplete number of the spectral pair $(\mu_{6,\{0,1,2\}},\Lambda_{6,\{0,1,2\}})$.
	\end{exam}
	
	The following example is due to Dutkay and Kraus
	\cite[Remark 2.14]{DK18}.
	
	\begin{exam}\label{exam2.13}
		For the case $b=6$, $N=2$ and let $p=55987$. There exists a non-zero integer cycle w.r.t. $(6,p\{0,1\})$. Terms of a period in order are
		$$311,9383,10895,11147,11189,11196,1866$$
		and then the prime number $p$ is an incomplete number of the spectral pair $(\mu_{6,\{0,3\}},\Lambda_{6,\{0,1\}})$.
	\end{exam}
	
	\noindent \textbf{Acknowledgements}
	
	The authors are grateful to Professor Xing-Gang He and Professor Li-Xiang An from Central China Normal University for many helpful discussions, comments, and suggestions during the preparation of this manuscript. The research is supported by the National Natural Science Foundation of China 12371087, 12171181 and 12361015.
	
	\noindent \textbf{Declaration of competing interest}
	
	The author declares no conflict of interest.

\end{document}